\documentclass[runningheads]{llncs}
\usepackage{amssymb}
\usepackage{amsmath}
\usepackage{indentfirst}
\usepackage[T1]{fontenc}
\usepackage{graphicx}

\begin{document}
	
\title{New classes of permutation polynomials with coefficients 1 over finite fields}

\author{Hutao Song\inst{1} \and
Hua Guo\inst{1,2} \and
Xiyong Zhang\inst{3} \and
Yapeng Wu\inst{1} \and
Jianwei Liu\inst{1}}

\authorrunning{Hutao Song et al.}

\institute{School of Cyber Science and Technology, Beihang University, Beijing, 100191, China \and
State Key Laboratory of Cryptology, P.O. Box 5159, Beijing, 100878, China \and
Beijing Institute of Satellite Information Engineering, Beijing, 100086, China
}

\maketitle  
         
\begin{abstract}
	
Permutation polynomials with coefficients 1 over finite fields attract researchers' interests due to their simple algebraic form. In this paper, we first construct four classes of fractional permutation polynomials over the cyclic subgroup of $ \mathbb{F}_{2^{2m}} $. From these permutation polynomials, three new classes of permutation polynomials with coefficients 1 over $ \mathbb{F}_{2^{2m}} $ are constructed, and three more general new classes of permutation polynomials with coefficients 1 over $ \mathbb{F}_{2^{2m}} $ are constructed using a new method we presented recently. Some known permutation polynomials are the special cases of our new permutation polynomials. Furthermore, we prove that, in all new permutation polynomials, there exists a permutation polynomial which is EA-inequivalent to known permutation polynomials for all even positive integer $ m $. This proof shows that EA-inequivalent permutation polynomials over $ \mathbb{F}_{q} $ can be constructed from EA-equivalent permutation polynomials over the cyclic subgroup of $ \mathbb{F}_{q} $. From this proof, it is obvious that, in all new permutation polynomials, there exists a permutation polynomial of which algebraic degree is the maximum algebraic degree of permutation polynomials over $ \mathbb{F}_{2^{2m}} $.


\end{abstract}

\section{Introduction}\label{intro}

Let $ q $ be a prime power, $ \mathbb{F}_{q} $ be the finite field with $ q $ elements, and $ f(x) $ be a polynomial over $ \mathbb{F}_{q} $. $ f(x) $ is a permutation polynomial over $ \mathbb{F}_{q} $ if $ f : x \mapsto f(x) $ is a permutation over $ \mathbb{F}_{q} $. For a positive integer $ d $, the cyclic subgroup $ \mu_{d}=\{x \in \overline{\mathbb{F}}_{q}: x^{d}=1\} $, where $ \overline{\mathbb{F}}_{q} $ is the algebraic closure of $\mathbb{F}_{q} $. Especially when $ q=2^{2m} $ for a positive integer $ m $ and $ d=2^{m}+1 $, $ \mu_{d} $ is denoted by $ U $. Permutation polynomials over finite fields have wide applications in areas of mathematics and engineering such as combinatorial designs \cite{Ref2}, coding theory \cite{Ref1,Ref4} and cryptography \cite{Ref8,Ref9}. Permutation polynomials with coefficients 1 have the simplest algebraic form, which can be used in cryptographic algorithms (such as multivariable public key cryptographic algorithm) to improve the efficiency. In recent years, permutation polynomials with coefficients 1 over finite fields attract researchers' interests \cite{Ref3,Ref5,Ref6,Ref7,Ref13,Ref14}.

Wan and Lidl \cite{Ref18} studied permutation polynomials of the form $ x^{r}f(x^{\frac{q-1}{d}}) $, and Zieve \cite{Ref16} presented a criterion for permutation polynomials of this form. In 2013, Tu et al. \cite{Ref10} studied permutation polynomials of the form $ \sum_{i=1}^{t}u_{i}x^{d_{i}} $, where $ u_{i} \in \mathbb{F}_{2^{2m}} $ and $ d_{i} \equiv e \pmod {2^{m}-1} $ with a positive integer $ e $, and they presented a criterion for permutation polynomials of this form. In both of these two criterions, permutation polynomials over $ \mathbb{F}_{q} $ are related with permutation polynomials over $ \mu_{d} $, and this motivates us to construct new permutation polynomials over $ \mathbb{F}_{q} $ from permutation polynomials over $ \mu_{d} $.

Many researchers studied permutation polynomials over $ \mu_{d} $ in recent years \cite{Ref3,Ref6,Ref7,Ref13,Ref14}, and these polynomials usually have fractional form. Li et al. \cite{Ref6} proved that $ \frac{x^{2^{k}+1}+x^{2^{k}}+1}{x^{2^{k}+1}+x+1} $ and $ \frac{x^{2^{k}+1}+x^{2^{k}}+x}{x^{2^{k}}+x+1} $ with a positive integer $ k $ are permutation polynomials over $ U $, and this motivates us to construct large classes of permutation polynomials.

In this paper, we first construct four classes of fractional permutation polynomials over $ U $. From these fractional permutation polynomials, three new classes of permutation polynomials with coefficients 1 over $ \mathbb{F}_{2^{2m}} $ are constructed by the criterion in \cite{Ref16}. Recently, we \cite{Ref20} presented a new method based on exponential sum and polar decomposition to construct permutation polynomials over $ \mathbb{F}_{q} $. By using this new method, three more general new classes of permutation polynomials with coefficients 1 over $ \mathbb{F}_{2^{2m}} $ are constructed from the four classes of fractional permutation polynomials mentioned above. Some known classes of permutation polynomials in \cite{Ref3,Ref6,Ref7,Ref13,Ref14} are the special cases of our new permutation polynomials.

Furthermore, by using the property that EA-equivalent nonconstant functions have the same algebraic degree, we prove that, in all new permutation polynomials, there exists a permutation polynomial which is EA-inequivalent to known permutation polynomials for all even positive integer $ m $. This proof shows that EA-inequivalent permutation polynomials over $ \mathbb{F}_{2^{2m}} $ can be constructed from EA-equivalent permutation polynomials over $ U $. Moreover, from this proof, it is obvious that, in all new permutation polynomials, there exists a permutation polynomial with algebraic degree $ 2m-1 $, which is the maximum algebraic degree of permutation polynomials over $ \mathbb{F}_{2^{2m}} $. Permutation polynomial with high alebraic degree can be used in cryptographic scheme to improve the security.

This paper is organized as follows. In Section \ref{sec:2}, some basic notations and lemmas are introduced. Section \ref{sec:3} constructs new classes of permutation polynomials with coefficients 1 over $ \mathbb{F}_{2^{2m}} $ from fractional permutation polynomials over $ U $, which is the main result of this paper. In Section \ref{sec:4}, we prove that, in all new permutation polynomials, there exists a permutation polynomial which is EA-inequivalent to known permutation polynomials for all even positive integer $ m $. Section \ref{sec:5} is the conclusion.

\section{Preliminaries}\label{sec:2}

First some basic notations and definitions are introduced.

Let $ v_{2}(\cdot) $ be the 2-adic order function.

Let $ n=2m $ for a positive integer $ m $, the trace function from $ \mathbb{F}_{2^n} $ to $ \mathbb{F}_{2} $ denoted by $ Tr_{1}^{n}(\cdot) $ is defined as $$ Tr_{1}^{n}(x)=x+x^{2}+x^{2^2}+\cdots+x^{2^{n-1}}. $$

\begin{definition}\label{def 3}{\rm\cite{Ref15}}
	Two functions $ F,G $ are called extended affine equivalent (EA-equivalent) if $ G=A_{1} \circ F \circ A_{2} + A $ for some affine permutations $ A_{1},A_{2} $ and affine function $ A $. 
\end{definition}

EA-equivalent nonconstant functions have the same algebraic degree. For any polynomial $ F(x)=\sum_{i=0}^{p^{n}-1}a_{i}x^{i} $ over $ \mathbb{F}_{p^n} $, the algebraic degree of $ F $ is defined as $ \deg(F)=\max_{0 \leq i \leq p^{n}-1}\left\{wt_{p}(i): a_{i} \neq 0\right\} $, where $ wt_{p}(i) $ is the $ p $-weight of $ i $ \cite{Ref19}. Especially when $ p=2 $, the $ p $-weight is Hamming weight.

The following three lemmas give two criterions for permutation polynomials, and both of them show that permutation polynomials over finite fields are related with permutation polynomials over $ \mu_{d} $.

\begin{lemma}\label{lemma 1}{\rm\cite{Ref16,Ref17}}	
	Pick $ d,r>0 $ with $ d|(q-1) $, and let $ h \in \mathbb{F}_{q}[x] $. Then $ f(x)=x^{r}h(x^{(q-1)/d}) $ permutes $ \mathbb{F}_{q} $ if and only if both
	
	(1) $ \gcd(r,(q-1)/d)=1 $ and
	
	(2) $ x^{r}h(x)^{(q-1)/d}$ permutes $ \mu_{d} $.
\end{lemma}

\begin{lemma}\label{lemma 2}{\rm\cite{Ref10}}	
	For an integer $ t \geq 2 $ and a positive integer $ e $, $ f(x)=x^{d_{1}}+\sum_{i=2}^{t}u_{i}x^{d_{i}} $, where for each $ i $ with $ 1 \leq i \leq t $, the element $ u_{i} \in \mathbb{F}_{2^{n}} $ and $ d_{i} \equiv e \pmod {2^{m}-1} $. Let $ \gcd(d_{1},2^{n}-1)=1 $, the polynomial $ f $ is a permutation polynomial over $ \mathbb{F}_{2^{n}} $ if and only if for every $ \delta \in \mathbb{F}_{2^{n}}^{*} $, $$ \sum_{x \in \mathbb{F}_{2^{n}}}(-1)^{Tr_{1}^{n}\left(x^{d_{1}}+\sum_{i=2}^{t}u_{i}\delta^{d_{1}-d_{i}}x^{d_{i}}\right)}=0. $$
\end{lemma}

\begin{lemma}\label{lemma 3}{\rm\cite{Ref10}}
	Let $ \gcd(d_{1},2^{n}-1)=1 $ and $ d_{i} \equiv e \pmod {2^{m}-1} $ for each $ i $ with $ 1 \leq i \leq t $. For $ \omega_{2},\cdots,\omega_{t} \in \mathbb{F}_{2^{n}} $, the exponential sum $$ \sum_{x \in \mathbb{F}_{2^{n}}}(-1)^{Tr_{1}^{n}\left(x^{d_{1}}+\sum_{i=2}^{t}\omega_{i}x^{d_{i}}\right)}=(N(\omega_{2},\cdots,\omega_{t})-1)\cdot2^{m}, $$ where $ N(\omega_{2},\cdots,\omega_{t}) $ is the number of $ \lambda's $ in $ U $ such that $$ \lambda^{d_{1}}+\sum_{i=2}^{t}\omega_{i}\lambda^{d_{i}}+\left(\lambda^{d_{1}}+\sum_{i=2}^{t}\omega_{i}\lambda^{d_{i}}\right)^{2^m}=0. $$
\end{lemma}

The following two lemmas give two classes of fractional permutation polynomials over $ U $.

\begin{lemma}\label{lemma 4}{\rm\cite{Ref6}}	
	For a positive integer $ k $, let $ \gcd(2^{k}-1,2^{m}+1)=1 $, then the fractional polynomial $ \frac{x^{2^{k}+1}+x^{2^k}+1}{x^{2^{k}+1}+x+1} $ is a permutation polynomial over $ U $.
\end{lemma}

\begin{lemma}\label{lemma 5}{\rm\cite{Ref6}}	
	For a positive integer $ k $, let $ \gcd(2^{k}+1,2^{m}+1)=1 $, then the fractional polynomial $ \frac{x^{2^{k}+1}+x^{2^k}+x}{x^{2^{k}}+x+1} $ is a permutation polynomial over $ U $.
\end{lemma}

\section{Constructions of new classes of permutation polynomials}\label{sec:3}

In this section, we construct four classes of fractional permutation polynomials over $ U $. From these fractional polynomials, six new classes of permutation polynomials with coefficients 1 over $ \mathbb{F}_{2^{2m}} $ are constructed.

\begin{lemma}\label{lemma 6}
	Let $ m $ be a positive even integer, $ k $ be a positive integer, $ s $ be an integer, and $ v_{2}(k) \leq v_{2}(m) $, then the fractional polynomial 
	\begin{equation}\label{eq 1}
		\frac{x^{2^{k}+2s+1}+x^{2^{k}+2s}+x^{2^{k}+s+1}+x^{2^{k}+s}+x^{2^{k}+1}+x^{2^k}+x^{2s}+x^{s}+1}{x^{2^{k}+2s+1}+x^{2^{k}+s+1}+x^{2^{k}+1}+x^{2s+1}+x^{2s}+x^{s+1}+x^{s}+x+1}
	\end{equation}
	is a class of permutation polynomials over $ U $.
\end{lemma}

\begin{proof}
	The fractional polynomial (\ref{eq 1}) can be written as
	\begin{equation*}
		\frac{(x^{2^{k}+1}+x^{2^k}+1)(x^{2s}+x^{s}+1)}{(x^{2^{k}+1}+x+1)(x^{2s}+x^{s}+1)}.
	\end{equation*}
	
	Obviously, $ x^{2}+x+1=0 $ has no solution in $ U $ if $ m $ is even, and $ \forall x \in U, x^{s} \in U $ if $ s $ is an integer. Thus, $ x^{2s}+x^{s}+1=0 $ has no solution in $ U $, then the fractional polynomial (\ref{eq 1}) can be written as
	\begin{equation*}
		\frac{x^{2^{k}+1}+x^{2^k}+1}{x^{2^{k}+1}+x+1}.
	\end{equation*}
	
	Since $ v_{2}(k) \leq v_{2}(m) $, $ \gcd(2^{k}-1,2^{m}+1)=1 $, according to Lemma \ref{lemma 4}, the fractional polynomial (\ref{eq 1}) is a class of permutation polynomials over $ U $. 
	
	The proof is complete.\qed
\end{proof}

\begin{theorem}\label{theorem 1}
	Let $ m $ be a positive even integer, $ k $ be a positive integer, $ s $ be an integer, $ v_{2}(k) \leq v_{2}(m) $, and $ \gcd(2^{k}+2s+1,2^{m}-1)=1 $, then the polynomial
	\begin{equation}\label{eq 4}
		\begin{split}	
			&x^{2^{k+m}+s2^{m+1}+2^{m}}+x^{2^{k+m}+s2^{m}+2^{m}+s}+x^{2^{k+m}+2^{m}+2s}+x^{s2^{m+1}+2^{m}+2^{k}}+\\
			&x^{s2^{m+1}+2^{k}+1}+x^{s2^{m}+2^{m}+2^{k}+s}+x^{s2^{m}+2^{k}+s+1}+x^{2^{m}+2^{k}+2s}+x^{2^{k}+2s+1}
		\end{split}
	\end{equation} 
	is a class of permutation polynomials over $ \mathbb{F}_{2^{2m}} $.
\end{theorem}

\begin{proof}
	The polynomial (\ref{eq 4}) can be written as $ x^{2^{k}+2s+1}h(x^{2^{m}-1}) $, where $ h(x)=x^{2^{k}+2s+1}+x^{2^{k}+s+1}+x^{2^{k}+1}+x^{2s+1}+x^{2s}+x^{s+1}+x^{s}+x+1 $. According to Lemma \ref{lemma 1}, the polynomial (\ref{eq 4}) permutes $ \mathbb{F}_{2^{2m}} $ if and only if $ x^{2^{k}+2s+1}h(x)^{2^{m}-1} $ permutes $ U $. For $ x \in U $,  $ x^{2^{k}+2s+1}h(x)^{2^{m}-1} $ can be written as
	\begin{equation*}
		\begin{split}
			&\frac{x^{2^{k}+2s+1}h(x)^{2^{m}}}{h(x)}=\frac{x^{2^{k}+2s+1}h(x^{2^{m}})}{h(x)}=\frac{x^{2^{k}+2s+1}h(x^{-1})}{h(x)}\\=
			&\frac{x^{2^{k}+2s+1}+x^{2^{k}+2s}+x^{2^{k}+s+1}+x^{2^{k}+s}+x^{2^{k}+1}+x^{2^k}+x^{2s}+x^{s}+1}{x^{2^{k}+2s+1}+x^{2^{k}+s+1}+x^{2^{k}+1}+x^{2s+1}+x^{2s}+x^{s+1}+x^{s}+x+1}.
		\end{split}
	\end{equation*}
	
	According to Lemma \ref{lemma 6}, $ x^{2^{k}+2s+1}h(x)^{2^{m}-1} $ permutes $ U $, thus the polynomial (\ref{eq 4}) is a class of permutation polynomials over $ \mathbb{F}_{2^{2m}} $. 
	
	The proof is complete.\qed
\end{proof}

Some known classes of pemutation polynomials are the special cases of the class of permutation polynomials (\ref{eq 4}):

1. (\cite{Ref3} Conjecture 2) $ k=2 $, $ s=0 $: 
\begin{equation*}
	x^{5}+x^{2^{m}+4}+x^{5\cdot2^{m}}.
\end{equation*}

2. (\cite{Ref14} Theorem 4.4) $ k=1 $, $ s=1 $: 
\begin{equation*}
	x^{5}+x^{4\cdot2^{m}+1}+x^{5\cdot2^{m}}.
\end{equation*}

\begin{theorem}\label{theorem 2}
	Let $ m $ be a positive even integer, $ k,i $ be positive integers, $ s,u $ be integers, $ Q=2^{m}+1 $, $ K=2^{k} $, $ I=2i $, $ v_{2}(k) \leq v_{2}(m) $, $ \gcd(i,Q)=1 $, $ \gcd(K+2s+1,Q)=1 $, and $ \gcd(d_{1},2^{2m}-1)=1 $, then the polynomial
	\begin{equation}\label{eq 5}
		\sum_{j=1}^{9}x^{d_{j}}, where 
		\left\{
		\begin{aligned}
			&d_{1}=(K+2s+1)i+uQ\\
			&d_{2}=(K+2s+1)i+uQ+(Q-2)i\\
			&d_{3}=(K+2s+1)i+uQ+(Q-2)si\\
			&d_{4}=(K+2s+1)i+uQ+(Q-2)(s+1)i\\
			&d_{5}=(K+2s+1)i+uQ+(Q-2)2si\\
			&d_{6}=(K+2s+1)i+uQ+(Q-2)(2s+1)i\\
			&d_{7}=(K+2s+1)i+uQ+(Q-2)(K+1)i\\
			&d_{8}=(K+2s+1)i+uQ+(Q-2)(K+s+1)i\\
			&d_{9}=(K+2s+1)i+uQ+(Q-2)(K+2s+1)i
		\end{aligned}
		\right.
	\end{equation}
	is a class of permutation polynomials over $ \mathbb{F}_{2^{2m}} $.
\end{theorem}	

\begin{proof}
	Since $ \gcd(i,Q)=1 $, $ x^I $ is a permutation polynomial over $ U $. According to Lemma \ref{lemma 6}, the fractional polynomial (\ref{eq 1}) permutes $ U $. We denote the fractional polynomial (\ref{eq 1}) by $ \alpha_{1}(x) $, then $ \alpha_{1}(x^I) $ permutes $ U $, thus, let $ \beta \in U $, the equation
	\begin{equation*}
		\alpha_{1}(x^I)=\beta
	\end{equation*}
	has only one solution in $ U $.
	
	For all $ x \in U $, $ x=\beta^{t}y $, where $ y \in U $ and $ t $ is an integer, thus the above equation can be written as
	\begin{equation}
		\begin{split}
			&\left[\beta^{(K+2s+1)tI}+\beta^{(K+2s+1)tI+1}\right]y^{(K+2s+1)I}+\beta^{(K+2s)tI}y^{(K+2s)I}\\+
			&\left[\beta^{(K+s+1)tI}+\beta^{(K+s+1)tI+1}\right]y^{(K+s+1)I}+\beta^{(K+s)tI}y^{(K+s)I}\\+
			&\left[\beta^{(K+1)tI}+\beta^{(K+1)tI+1}\right]y^{(K+1)I}+\beta^{KtI}y^{KI}+\beta^{(2s+1)tI+1}y^{(2s+1)I}\\+
			&\left[\beta^{2stI}+\beta^{2stI+1}\right]y^{2sI}+\beta^{(s+1)tI+1}y^{(s+1)I}\\+
			&\left[\beta^{stI}+\beta^{stI+1}\right]y^{sI}+\beta^{tI+1}y^{I}+\beta+1=0.
		\end{split}
	\end{equation}
	
	Since $ \gcd(i,Q)=1 $ and $ \gcd(K+2s+1,Q)=1 $, there exists $ t \equiv \frac{1}{KI+2sI+I} \pmod Q $. Dividing both sides of the above equation by $ \beta^{(K+2s+1)tI}y^{(K+2s+1)i} $, then the equation
	\begin{equation*}
		g(y)+g(y)^{2^{m}}=0
	\end{equation*}
	has only one solution in $ U $, where
	\begin{equation*}
		\begin{split}
			g(y)=
			&y^{(K+2s+1)i}+\beta^{-tI}y^{(K+2s-1)i}+\beta^{-stI}y^{(K+1)i}\\+
			&\beta^{-(s+1)tI}y^{(K-1)i}+\beta^{-2stI}y^{(K-2s+1)i}\\+
			&\beta^{-(2s+1)tI}y^{(K-2s-1)i}+\beta^{-(K+1)tI}y^{(-K+2s-1)i}\\+
			&\beta^{-(K+s+1)tI}y^{(-K-1)i}+\beta^{-(K+2s+1)tI}y^{(-K-2s-1)i}.
		\end{split}
	\end{equation*}
	
	Let $ a $ be an integer, then $ \beta $ can be written as $ \delta^{(Q-2)a} $ where $ \delta \in \mathbb{F}_{2^{n}}^{*} $. According to the above equation, the equation
	\begin{equation*}
		g(\lambda)+g(\lambda)^{2^{m}}=0
	\end{equation*}
	has only one solution in $ U $, where
	\begin{equation*}
		\begin{split}
			g(\lambda)=&\lambda^{(K+2s+1)i}+\delta^{-(Q-2)atI}\lambda^{(K+2s-1)i}+\delta^{-(Q-2)astI}\lambda^{(K+1)i}\\+
			&\delta^{-(Q-2)(s+1)atI}\lambda^{(K-1)i}+\delta^{-(Q-2)2satI}\lambda^{(K-2s+1)i}\\+
			&\delta^{-(Q-2)(2s+1)atI}\lambda^{(K-2s-1)i}+\delta^{-(Q-2)(K+1)atI}\lambda^{(-K+2s-1)i}\\+
			&\delta^{-(Q-2)(K+s+1)atI}\lambda^{(-K-1)i}+\delta^{-(Q-2)(K+2s+1)atI}\lambda^{(-K-2s-1)i}.
		\end{split}
	\end{equation*}
	
	Inspired by Lemma \ref{lemma 2} and Lemma \ref{lemma 3}, we hope that $ g(\lambda) $ has the form of $ \lambda^{d_{1}}+\sum_{j=2}^{9}\delta^{d_{1}-d_{j}}\lambda^{d_{j}} $. Let $ u_{l} $ be integers with $ l=1,\cdots,8 $, we assume that
	
	\begin{equation}\label{eq 48}
		\left\{
		\begin{aligned}
			&d_{1} \equiv (K+2s+1)i \pmod Q \\ 
			&d_{2} \equiv (K+2s-1)i \pmod Q \\
			&d_{3} \equiv (K+1)i \pmod Q \\
			&d_{4} \equiv (K-1)i \pmod Q \\
			&d_{5} \equiv (K-2s+1)i \pmod Q \\
			&d_{6} \equiv (K-2s-1)i \pmod Q \\
			&d_{7} \equiv (-K+2s-1)i \pmod Q \\
			&d_{8} \equiv (-K-1)i \pmod Q \\
			&d_{9} \equiv (-K-2s-1)i \pmod Q\\
			&d_{1}-d_{2}=-(Q-2)atI+u_{1}Q(Q-2) \\ 
			&d_{1}-d_{3}=-(Q-2)astI+u_{2}Q(Q-2) \\
			&d_{1}-d_{4}=-(Q-2)(s+1)atI+u_{3}Q(Q-2) \\
			&d_{1}-d_{5}=-(Q-2)2satI+u_{4}Q(Q-2) \\
			&d_{1}-d_{6}=-(Q-2)(2s+1)atI+u_{5}Q(Q-2) \\
			&d_{1}-d_{7}=-(Q-2)(K+1)atI+u_{6}Q(Q-2) \\
			&d_{1}-d_{8}=-(Q-2)(K+s+1)atI+u_{7}Q(Q-2) \\
			&d_{1}-d_{9}=-(Q-2)(K+2s+1)atI+u_{8}Q(Q-2).
		\end{aligned}
		\right.
	\end{equation}
	
	Substituting $ d_{1}=(K+2s+1)i+uQ $ into the last eight equations of (\ref{eq 48}), we have
	\begin{equation}\label{eq 49}
		\left\{
		\begin{aligned}
			&d_{2}=(K+2s+1)i+uQ+(Q-2)atI-u_{1}Q(Q-2) \\ 
			&d_{3}=(K+2s+1)i+uQ+(Q-2)astI-u_{2}Q(Q-2) \\
			&d_{4}=(K+2s+1)i+uQ+(Q-2)(s+1)atI-u_{3}Q(Q-2) \\
			&d_{5}=(K+2s+1)i+uQ+(Q-2)2satI-u_{4}Q(Q-2) \\
			&d_{6}=(K+2s+1)i+uQ+(Q-2)(2s+1)atI-u_{5}Q(Q-2) \\
			&d_{7}=(K+2s+1)i+uQ+(Q-2)(K+1)atI-u_{6}Q(Q-2) \\
			&d_{8}=(K+2s+1)i+uQ+(Q-2)(K+s+1)atI-u_{7}Q(Q-2) \\
			&d_{9}=(K+2s+1)i+uQ+(Q-2)(K+2s+1)atI-u_{8}Q(Q-2).
		\end{aligned}
		\right.
	\end{equation}
	
	Substituting (\ref{eq 49}) into the first nine equations of (\ref{eq 48}), we obtain
	\begin{equation}\label{amend 1}
		\left\{
		\begin{aligned}
			&(K+2s+1)i-2atI \equiv (K+2s-1)i \pmod Q  \\ 
			&(K+2s+1)i-2astI \equiv (K+1)i \pmod Q \\
			&(K+2s+1)i-2(s+1)atI \equiv (K-1)i \pmod Q \\
			&(K+2s+1)i-4satI \equiv (K-2s+1)i \pmod Q \\
			&(K+2s+1)i-2(2s+1)atI \equiv (K-2s-1)i \pmod Q \\
			&(K+2s+1)i-2(K+1)atI \equiv (-K+2s-1)i \pmod Q \\
			&(K+2s+1)i-2(K+s+1)atI \equiv (-K-1)i \pmod Q \\
			&(K+2s+1)i-2(K+2s+1)atI) \equiv (-K-2s-1)i \pmod Q.
		\end{aligned}
		\right.
	\end{equation}
	
	We can conclude $ a \equiv \frac{1}{2t} \pmod Q $ from the equations (\ref{amend 1}). Substituting $ a \equiv \frac{1}{2t} \pmod Q $ into the equations (\ref{eq 49}), we have
	
	\begin{equation}\label{amend 2}
		\left\{
		\begin{aligned}
			d_{2} &\equiv (K+2s+1)i+uQ+(Q-2)i \pmod {Q(Q-2)} \\ 
			d_{3} &\equiv (K+2s+1)i+uQ+(Q-2)si \pmod {Q(Q-2)} \\
			d_{4} &\equiv (K+2s+1)i+uQ+(Q-2)(s+1)i \pmod {Q(Q-2)} \\
			d_{5} &\equiv (K+2s+1)i+uQ+(Q-2)2si \pmod {Q(Q-2)} \\
			d_{6} &\equiv (K+2s+1)i+uQ+(Q-2)(2s+1)i \pmod {Q(Q-2)} \\
			d_{7} &\equiv (K+2s+1)i+uQ+(Q-2)(K+1)i \pmod {Q(Q-2)} \\
			d_{8} &\equiv (K+2s+1)i+uQ+(Q-2)(K+s+1)i \pmod {Q(Q-2)} \\
			d_{9} &\equiv (K+2s+1)i+uQ+(Q-2)(K+2s+1)i \pmod {Q(Q-2)}.
		\end{aligned}
		\right.
	\end{equation}
	
	Thus, the equation
	\begin{equation*}
		\lambda^{d_{1}}+\sum_{j=2}^{9}\delta^{d_{1}-d_{j}}\lambda^{d_{j}}+\left(\lambda^{d_{1}}+\sum_{j=2}^{9}\delta^{d_{1}-d_{j}}\lambda^{d_{j}}\right)^{2^m}=0
	\end{equation*}
	has only one solution in $ U $, where $ d_{1}=(K+2s+1)i+uQ $ and $ d_{j} $($ j=2,\cdots,9 $) are given by (\ref{amend 2}). 
	
	According to Lemma \ref{lemma 2} and Lemma \ref{lemma 3}, $ \sum_{j=1}^{9}x^{d_{j}} $ is a class of permutation polynomials over $\mathbb{F}_{2^{2m}} $. For $ x \in \mathbb{F}_{2^{2m}} $, $ x^{d_{j}}=x^{d_{j} \pmod {Q(Q-2})} $ since $ Q(Q-2)=2^{2m}-1 $. Thus, the polynomial (\ref{eq 5}) is a class of permutation polynomials over $ \mathbb{F}_{2^{2m}} $. 
	
	The proof is complete.\qed
\end{proof}

The polynomial (\ref{eq 5}) is a large class of permutation polynomials. Some known classes of pemutation polynomials are the special cases of the class of permutation polynomials (\ref{eq 5}):

1. (\cite{Ref7} Theorem 3) $ k=2 $, $ s=0 $, $ u=-i $, $ i \equiv 1/3 \pmod {2^{m}+1} $: 
\begin{equation*}
	x+x^{1-(2^{m}-1)/3}+x^{1+4(2^{m}-1)/3}.
\end{equation*}

2. (\cite{Ref6} Theorem 1) $ s=0 $, $ u=0 $, $ i \equiv 1/(2^{k}+1) \pmod {2^{m}+1} $: 
\begin{equation*}
	x+x^{2^{m}}+x^{1+(2^{m}-1)/(2^{k}+1)}.
\end{equation*}

3. (Theorem \ref{theorem 1}) $ u=0 $, $ i=1 $: 
\begin{equation*}
	\begin{split}	
		&x^{2^{k+m}+s2^{m+1}+2^{m}}+x^{2^{k+m}+s2^{m}+2^{m}+s}+x^{2^{k+m}+2^{m}+2s}+x^{s2^{m+1}+2^{m}+2^{k}}+\\
		&x^{s2^{m+1}+2^{k}+1}+x^{s2^{m}+2^{m}+2^{k}+s}+x^{s2^{m}+2^{k}+s+1}+x^{2^{m}+2^{k}+2s}+x^{2^{k}+2s+1}.
	\end{split}
\end{equation*}

(i) (\cite{Ref3} Conjecture 2) $ k=2 $, $ s=0 $: 
\begin{equation*}
	x^{5}+x^{2^{m}+4}+x^{5\cdot2^{m}}.
\end{equation*}

(ii) (\cite{Ref14} Theorem 4.4) $ k=1 $, $ s=1 $: 
\begin{equation*}
	x^{5}+x^{4\cdot2^{m}+1}+x^{5\cdot2^{m}}.
\end{equation*}

\begin{lemma}\label{lemma 7}
	Let $ m $ be a positive even integer, $ k $ be a positive integer, $ s $ be an integer, and $ \gcd(2^{k}+1,2^{m}+1)=1 $, then the fractional polynomial 
	\begin{equation}\label{eq 14}
		\frac{x^{2^{k}+2s+1}+x^{2^{k}+2s}+x^{2^{k}+s+1}+x^{2^{k}+s}+x^{2^{k}+1}+x^{2^k}+x^{2s+1}+x^{s+1}+x}{x^{2^{k}+2s}+x^{2^{k}+s}+x^{2^{k}}+x^{2s+1}+x^{2s}+x^{s+1}+x^{s}+x+1}
	\end{equation}
	is a class of permutation polynomials over $ U $.
\end{lemma}

\begin{proof}
	The fractional polynomial (\ref{eq 14}) can be written as
	\begin{equation*}
		\frac{(x^{2^{k}+1}+x^{2^k}+x)(x^{2s}+x^{s}+1)}{(x^{2^{k}}+x+1)(x^{2s}+x^{s}+1)}.
	\end{equation*}
	
	Obviously, $ x^{2}+x+1=0 $ has no solution in $ U $ if $ m $ is even, and $ \forall x \in U, x^{s} \in U $ if $ s $ is an integer. Thus, $ x^{2s}+x^{s}+1=0 $ has no solution in $ U $, then the fractional polynomial (\ref{eq 14}) can be written as
	\begin{equation*}
		\frac{x^{2^{k}+1}+x^{2^k}+x}{x^{2^{k}}+x+1}.
	\end{equation*}
	
	According to Lemma \ref{lemma 5}, the fractional polynomial (\ref{eq 14}) is a class of permutation polynomials over $ U $. 
	
	The proof is complete. \qed
\end{proof}

\begin{theorem}\label{theorem 3}
	Let $ m $ be a positive even integer, $ k $ be a positive integer, $ s $ be an integer, $ \gcd(2^{k}+1,2^{m}+1)=1 $, and $ \gcd(2^{k}+2s+1,2^{m}-1)=1 $, then the polynomial
	\begin{equation}\label{eq 17}
		\begin{split}	
			&x^{2^{k+m}+s2^{m+1}+1}+x^{2^{k+m}+s2^{m}+s+1}+x^{2^{k+m}+2s+1}+x^{s2^{m+1}+2^{m}+2^{k}}+\\
			&x^{s2^{m+1}+2^{k}+1}+x^{s2^{m}+2^{m}+2^{k}+s}+x^{s2^{m}+2^{k}+s+1}+x^{2^{m}+2^{k}+2s}+x^{2^{k}+2s+1}
		\end{split}
	\end{equation} 
	is a class of permutation polynomials over $ \mathbb{F}_{2^{2m}} $.
\end{theorem}

\begin{proof}
	The polynomial (\ref{eq 17}) can be written as $ x^{2^{k}+2s+1}h(x^{2^{m}-1}) $, where $ h(x)=x^{2^{k}+2s}+x^{2^{k}+s}+x^{2^{k}}+x^{2s+1}+x^{2s}+x^{s+1}+x^{s}+x+1 $. According to Lemma \ref{lemma 1}, the polynomial (\ref{eq 17}) permutes $ \mathbb{F}_{2^{2m}} $ if and only if $ x^{2^{k}+2s+1}h(x)^{2^{m}-1} $ permutes $ U $. For $ x \in U $, $ x^{2^{k}+2s+1}h(x)^{2^{m}-1} $ can be written as the fractional polynomial (\ref{eq 14}). According to Lemma \ref{lemma 7}, $ x^{2^{k}+2s+1}h(x)^{2^{m}-1} $ permutes $ U $, thus the polynomial (\ref{eq 17}) is a class of permutation polynomials over $ \mathbb{F}_{2^{2m}} $. 
	
	The proof is complete. \qed
\end{proof}

A known class of pemutation polynomials is the special case of the class of permutation polynomials (\ref{eq 17}):

1. (\cite{Ref13} Theorem 4.2) $ k=2 $, $ s=1 $: 
\begin{equation*}
	x^{7}+x^{3\cdot2^{m}+4}+x^{4\cdot2^{m}+3}+x^{5\cdot2^{m}+2}+x^{6\cdot2^{m}+1}.
\end{equation*}

\begin{theorem}\label{theorem 4}
	Let $ m $ be a positive even integer, $ k,i $ be positive integers, $ s,u $ be integers, $ Q=2^{m}+1 $, $ K=2^{k} $, $ I=2i $, $ \gcd(i,Q)=1 $, $ \gcd(K+1,Q)=1 $, $ \gcd(K+2s+1,Q)=1 $, and $ \gcd(d_{1},2^{2m}-1)=1 $, then the polynomial
	\begin{equation}\label{eq 18}
		\sum_{j=1}^{9}x^{d_{j}}, where 
		\left\{
		\begin{aligned}
			&d_{1}=(K+2s+1)i+uQ\\
			&d_{2}=(K+2s+1)i+uQ+(Q-2)i\\
			&d_{3}=(K+2s+1)i+uQ+(Q-2)si\\
			&d_{4}=(K+2s+1)i+uQ+(Q-2)(s+1)i\\
			&d_{5}=(K+2s+1)i+uQ+(Q-2)2si\\
			&d_{6}=(K+2s+1)i+uQ+(Q-2)(2s+1)i\\
			&d_{7}=(K+2s+1)i+uQ+(Q-2)Ki\\
			&d_{8}=(K+2s+1)i+uQ+(Q-2)(K+s)i\\
			&d_{9}=(K+2s+1)i+uQ+(Q-2)(K+2s)i
		\end{aligned}
		\right.
	\end{equation}
	is a class of permutation polynomials over $ \mathbb{F}_{2^{2m}} $.
\end{theorem}	

\begin{proof}
	Since $ \gcd(i,Q)=1 $, $ x^I $ is a permutation polynomial over $ U $. According to Lemma \ref{lemma 7}, the fractional polynomial (\ref{eq 14}) permutes $ U $. We denote the fractional polynomial (\ref{eq 14}) by $ \alpha_{2}(x) $, then $ \alpha_{2}(x^I) $ permutes $ U $, thus, let $ \beta \in U $, the equation
	\begin{equation*}
		\alpha_{2}(x^I)=\beta
	\end{equation*}
	has only one solution in $ U $.
	
	For all $ x \in U $, $ x=\beta^{t}y $, where $ y \in U $ and $ t $ is an integer, thus the above equation can be written as
	\begin{equation*}
		\begin{split}
			&\beta^{(K+2s+1)tI}y^{(K+2s+1)I}+\left[\beta^{(K+2s)tI}+\beta^{(K+2s)tI+1}\right]y^{(K+2s)I}\\+
			&\beta^{(K+s+1)tI}y^{(K+s+1)I}+\left[\beta^{(K+s)tI}+\beta^{(K+s)tI+1}\right]y^{(K+s)I}\\+
			&\beta^{(K+1)tI}y^{(K+1)I}+\left[\beta^{KtI}+\beta^{KtI+1}\right]y^{KI}\\+
			&\left[\beta^{(2s+1)tI}+\beta^{(2s+1)tI+1}\right]y^{(2s+1)I}+\beta^{2stI+1}y^{2sI}\\+
			&\left[\beta^{(s+1)tI}+\beta^{(s+1)tI+1}\right]y^{(s+1)I}+\beta^{stI+1}y^{sI}\\+
			&\left[\beta^{tI}+\beta^{tI+1}\right]y^{I}+\beta=0.
		\end{split}
	\end{equation*}
	
	Since $ \gcd(i,Q)=1 $ and $ \gcd(K+2s+1,Q)=1 $, there exists $ t \equiv \frac{1}{KI+2sI+I} \pmod Q $. Dividing both sides of the above equation by $ \beta^{(K+2s+1)tI}y^{(K+2s+1)i} $, then the equation
	\begin{equation*}
		g(y)+g(y)^{2^{m}}=0
	\end{equation*}
	has only one solution in $ U $, where
	\begin{equation*}
		\begin{split}
			g(y)=
			&y^{(K+2s+1)i}+\beta^{-tI}y^{(K+2s-1)i}+\beta^{-stI}y^{(K+1)i}\\+
			&\beta^{-(s+1)tI}y^{(K-1)i}+\beta^{-2stI}y^{(K-2s+1)i}\\+
			&\beta^{-(2s+1)tI}y^{(K-2s-1)i}+\beta^{-KtI}y^{(-K+2s+1)i}\\+
			&\beta^{-(K+s)tI}y^{(-K+1)i}+\beta^{-(K+2s)tI}y^{(-K-2s+1)i}.
		\end{split}
	\end{equation*}
	
	Let $ a $ be an integer, then $ \beta $ can be written as $ \delta^{(Q-2)a} $ where $ \delta \in \mathbb{F}_{2^{n}}^{*} $. According to the above equation, the equation
	\begin{equation*}
		g(\lambda)+g(\lambda)^{2^{m}}=0
	\end{equation*}
	has only one solution in $ U $, where
	\begin{equation*}
		\begin{split}
			g(\lambda)=
			&\lambda^{(K+2s+1)i}+\delta^{-(Q-2)atI}\lambda^{(K+2s-1)i}+\delta^{-(Q-2)astI}\lambda^{(K+1)i}\\+
			&\delta^{-(Q-2)(s+1)atI}\lambda^{(K-1)i}+\delta^{-(Q-2)2satI}\lambda^{(K-2s+1)i}\\+
			&\delta^{-(Q-2)(2s+1)atI}\lambda^{(K-2s-1)i}+\delta^{-(Q-2)KatI}\lambda^{(-K+2s+1)i}\\+
			&\delta^{-(Q-2)(K+s)atI}\lambda^{(-K+1)i}+\delta^{-(Q-2)(K+2s)atI}\lambda^{(-K-2s+1)i}.
		\end{split}
	\end{equation*}
	
	Inspired by Lemma \ref{lemma 2} and Lemma \ref{lemma 3}, we hope that $ g(\lambda) $ has the form of $ \lambda^{d_{1}}+\sum_{j=2}^{9}\delta^{d_{1}-d_{j}}\lambda^{d_{j}} $. Similar to the proof of Theorem \ref{theorem 2}, we can conclude that the equation
	\begin{equation*}
		\lambda^{d_{1}}+\sum_{j=2}^{9}\delta^{d_{1}-d_{j}}\lambda^{d_{j}}+\left(\lambda^{d_{1}}+\sum_{j=2}^{9}\delta^{d_{1}-d_{j}}\lambda^{d_{j}}\right)^{2^m}=0
	\end{equation*}
	has only one solution in $ U $, where
	\begin{equation*}
		\left\{
		\begin{aligned}
			&d_{1} \equiv (K+2s+1)i+uQ \pmod{Q(Q-2)}\\
			&d_{2} \equiv (K+2s+1)i+uQ+(Q-2)i \pmod{Q(Q-2)}\\
			&d_{3} \equiv (K+2s+1)i+uQ+(Q-2)si \pmod{Q(Q-2)}\\
			&d_{4} \equiv (K+2s+1)i+uQ+(Q-2)(s+1)i \pmod{Q(Q-2)}\\
			&d_{5} \equiv (K+2s+1)i+uQ+(Q-2)2si \pmod{Q(Q-2)}\\
			&d_{6} \equiv (K+2s+1)i+uQ+(Q-2)(2s+1)i \pmod{Q(Q-2)}\\
			&d_{7} \equiv (K+2s+1)i+uQ+(Q-2)Ki \pmod{Q(Q-2)}\\
			&d_{8} \equiv (K+2s+1)i+uQ+(Q-2)(K+s)i \pmod{Q(Q-2)}\\
			&d_{9} \equiv (K+2s+1)i+uQ+(Q-2)(K+2s)i \pmod{Q(Q-2)}
		\end{aligned}
		\right.
	\end{equation*}
	
	According to Lemma \ref{lemma 2} and Lemma \ref{lemma 3}, $ \sum_{j=1}^{9}x^{d_{j}} $ is a class of permutation polynomials over $ \mathbb{F}_{2^{2m}} $. For $ x \in \mathbb{F}_{2^{2m}} $, $ x^{d_{j}}=x^{d_{j} \pmod {Q(Q-2})} $ since $ Q(Q-2)=2^{2m}-1 $. Thus, the polynomial (\ref{eq 18}) is a class of permutation polynomials over $ \mathbb{F}_{2^{2m}} $. 
	
	The proof is complete. \qed
\end{proof}

The polynomial (\ref{eq 18}) is a large class of permutation polynomials. Some known classes of pemutation polynomials are the special cases of the class of permutation polynomials (\ref{eq 18}):

1. (\cite{Ref14} Theorem 3.2) $ k=1 $, $ s=0 $, $ u=1 $, $ i=1 $: 
\begin{equation*}
	x^{2^{m}+4}+x^{2^{m+1}+3}+x^{5\cdot2^{m}}.
\end{equation*}

2. (\cite{Ref13} Theorem 4.9) $ k=3 $, $ s=0 $, $ u=-1 $, $ i=2^{m} $: 
\begin{equation*}
	x^{7}+x^{7\cdot2^{m}}+x^{8\cdot2^{m}-1}.
\end{equation*}

3. (\cite{Ref6} Theorem 2) $ s=0 $, $ u=0 $, $ i \equiv 1/(2^{k}+1) \pmod {2^{m}+1} $: 
\begin{equation*}
	x+x^{1+(2^{m}-1)/(2^{k}+1)}+x^{1+2^{k}(2^{m}-1)/(2^{k}+1)}.
\end{equation*}

(i) (\cite{Ref7} Theorem 6) $ k=2 $:
\begin{equation*}
	x+x^{1+(2^{m}-1)/5}+x^{1+4(2^{m}-1)/5}.
\end{equation*}

4. (Theorem \ref{theorem 3}) $ u=0 $, $ i=1 $: 
\begin{equation*}
	\begin{split}	
		&x^{2^{k+m}+s2^{m+1}+1}+x^{2^{k+m}+s2^{m}+s+1}+x^{2^{k+m}+2s+1}+x^{s2^{m+1}+2^{m}+2^{k}}+\\
		&x^{s2^{m+1}+2^{k}+1}+x^{s2^{m}+2^{m}+2^{k}+s}+x^{s2^{m}+2^{k}+s+1}+x^{2^{m}+2^{k}+2s}+x^{2^{k}+2s+1}.
	\end{split}
\end{equation*}

(i) (\cite{Ref13} Theorem 4.2) $ k=2 $, $ s=1 $: 
\begin{equation*}
	x^{7}+x^{3\cdot2^{m}+4}+x^{4\cdot2^{m}+3}+x^{5\cdot2^{m}+2}+x^{6\cdot2^{m}+1}.
\end{equation*}

\begin{lemma}\label{lemma 8}
	Let $ m,k $ be positive even integers, and $ v_{2}(k) \leq v_{2}(m) $, then the fractional polynomial 
	\begin{equation}\label{eq 27}
		\frac{\sum_{i=1}^{2^{k}}x^{2^{k}-i}}{\sum_{j=1}^{2^{k}}x^{2^{k}-j}}
	\end{equation}
	with $ i \equiv 0,1 \pmod 3 $ and $ j \equiv 1,2 \pmod 3 $ is a class of permutation polynomials over $ U $.
\end{lemma}

\begin{proof}
	Since $ k $ is even, $ 2^{k} \equiv 1 \pmod 3 $. Thus, $ x^{2^{k}} \equiv x \pmod {x^{3}+1} $ over $ \mathbb{F}_{2^{2m}} $. Therefore,
	\begin{equation*}
		\begin{split}
			&\gcd(x^{2^{k}+1}+x^{2^{k}}+1,x^{2^{k}+1}+x+1)\\=
			&\gcd(x^{2^{k}}+x,x^{2^{k}+1}+x+1)\\=
			&\gcd(x^{2^{k}}+x,x(x^{2^{k}}+x)+x^{2}+x+1)\\=
			&\gcd(x^{2^{k}}+x,x^{2}+x+1)\\=
			&\gcd(x^{3}+1,x^{2}+x+1)\\=
			&x^{2}+x+1.
		\end{split}
	\end{equation*}
	
	According to the above equation, we have
	\begin{equation*}
		\frac{x^{2^{k}+1}+x^{2^{k}}+1}{x^{2^{k}+1}+x+1}=\frac{(x^{2}+x+1)\sum_{i=1}^{2^{k}}x^{2^{k}-i}}{(x^{2}+x+1)\sum_{j=1}^{2^{k}}x^{2^{k}-j}},
	\end{equation*}
	where $ i \equiv 0,1 \pmod 3 $ and $ j \equiv 1,2 \pmod 3 $.
	
	Obviously, $ x^{2}+x+1=0 $ has no solution in $ U $ if $ m $ is even, thus
	\begin{equation*}
		\frac{x^{2^{k}+1}+x^{2^{k}}+1}{x^{2^{k}+1}+x+1}=\frac{\sum_{i=1}^{2^{k}}x^{2^{k}-i}}{\sum_{j=1}^{2^{k}}x^{2^{k}-j}},
	\end{equation*}
	where $ i \equiv 0,1 \pmod 3 $ and $ j \equiv 1,2 \pmod 3 $.
	
	According to Lemma \ref{lemma 4}, the fractional polynomial (\ref{eq 27}) is a class of permutation polynomials over $ U $.
	
	The proof is complete.\qed
\end{proof}	

\begin{lemma}\label{lemma 9}
	Let $ m $ be a positive even integer, $ k $ be a positive odd integer, and $ \gcd(2^{k}+1,2^{m}+1)=1 $, then the fractional polynomial 
	\begin{equation}\label{eq 31}
		\frac{\sum_{i=1}^{2^{k}}x^{2^{k}-i}}{\sum_{j=1}^{2^{k}}x^{2^{k}-j}}
	\end{equation}
	with $ i \equiv 0,1 \pmod 3 $ and $ j \equiv 0,2 \pmod 3 $ is a class of permutation polynomials over $ U $.
\end{lemma}

\begin{proof}
	Since $ k $ is odd, $ 2^{k} \equiv 2 \pmod 3 $. Thus, $ x^{2^{k}} \equiv x^{2} \pmod {x^{3}+1} $ over $ \mathbb{F}_{2^{2m}} $. Therefore,
	\begin{equation*}
		\begin{split}
			&\gcd(x^{2^{k}+1}+x^{2^{k}}+x,x^{2^{k}}+x+1)\\=
			&\gcd(x^{2^{k}+1}+1,x^{2^{k}}+x+1)\\=
			&\gcd(x(x^{2^{k}}+x+1)+x^{2}+x+1,x^{2^{k}}+x+1)\\=
			&\gcd(x^{2}+x+1,x^{2^{k}}+x+1)\\=
			&\gcd(x^{2}+x+1,x^{2}+x+1)\\=
			&x^{2}+x+1.
		\end{split}
	\end{equation*}
	
	According to the above equation, we have
	\begin{equation*}
		\frac{x^{2^{k}+1}+x^{2^{k}}+x}{x^{2^{k}}+x+1}=\frac{(x^{2}+x+1)\sum_{i=1}^{2^{k}}x^{2^{k}-i}}{(x^{2}+x+1)\sum_{j=1}^{2^{k}}x^{2^{k}-j}},
	\end{equation*}
	where $ i \equiv 0,1 \pmod 3 $ and $ j \equiv 0,2 \pmod 3 $.
	
	Obviously, $ x^{2}+x+1=0 $ has no solution in $ U $ if $ m $ is even, thus
	\begin{equation*}
		\frac{x^{2^{k}+1}+x^{2^{k}}+x}{x^{2^{k}}+x+1}=\frac{\sum_{i=1}^{2^{k}}x^{2^{k}-i}}{\sum_{j=1}^{2^{k}}x^{2^{k}-j}},
	\end{equation*}
	where $ i \equiv 0,1 \pmod 3 $ and $ j \equiv 0,2 \pmod 3 $.
	
	According to Lemma \ref{lemma 5}, the fractional polynomial (\ref{eq 31}) is a class of permutation polynomials over $ U $. 
	
	The proof is complete.\qed
\end{proof}	

\begin{theorem}\label{theorem 5}
	Let $ m $ be a positive even integer, $ k $ be a positive odd integer, $ \gcd(2^{k}-1,2^{m}-1)=1 $, $ \gcd(2^{k}+1,2^{m}+1)=1 $, and $ j \equiv 0,2 \pmod 3 $, then the polynomial
	\begin{equation}\label{eq 35}
		\sum_{j=1}^{2^{k}}x^{2^{k+m}-j2^{m}+j-1}
	\end{equation} 
	is a class of permutation polynomials over $ \mathbb{F}_{2^{2m}} $.
\end{theorem}

\begin{proof}
	The polynomial (\ref{eq 35}) can be written as $ x^{2^{k}-1}h(x^{2^{m}-1}) $, where $ h(x)=\sum_{j=1}^{2^{k}}x^{2^{k}-j} $. According to Lemma \ref{lemma 1}, the polynomial (\ref{eq 35}) permutes $ \mathbb{F}_{2^{2m}} $ if and only if $ x^{2^{k}-1}h(x)^{2^{m}-1} $ permutes $ U $. Since $ k $ is odd, for $ x \in U $, $ x^{2^{k}-1}h(x)^{2^{m}-1} $ can be written as the fractional polynomial (\ref{eq 31}).
	
	According to Lemma \ref{lemma 9}, $ x^{2^{k}-1}h(x)^{2^{m}-1} $ permutes $ U $, thus the polynomial (\ref{eq 35}) is a class of permutation polynomials over $ \mathbb{F}_{2^{2m}} $.
	
	The proof is complete.\qed
\end{proof}

A known class of pemutation polynomials is the special case of the class of permutation polynomials (\ref{eq 35}):

1. (\cite{Ref13} Theorem 3.5) $ k=3 $: 
\begin{equation*}
	x^{7}+x^{2\cdot2^{m}+5}+x^{3\cdot2^{m}+4}+x^{5\cdot2^{m}+2}+x^{6\cdot2^{m}+1}.
\end{equation*}

\begin{theorem}\label{theorem 6}
	Let $ m $ be a positive even integer, $ k,i $ be positive integers, $ u $ be an integer, $ Q=2^{m}+1 $, $ K=2^{k} $, $ I=2i $, $ v_{2}(k) \leq v_{2}(m) $, $ \gcd(i,Q)=1 $, and $ \gcd(d_{1},2^{2m}-1)=1 $, then the polynomial
	\begin{equation}\label{eq 36}
		\sum_{j=1}^{K}x^{d_{j}}, where \ d_{j}=(K-1)i+uQ+(Q-2)(j-1)i \ with \ j \equiv 0,1 \pmod 3 
	\end{equation}
	is a class of permutation polynomials over $ \mathbb{F}_{2^{2m}} $ if one of the following conditions is satisfied:
	
	(i) $ k $ is even or
	
	(ii) $ k $ is odd and $ \gcd(K+1,Q)=1 $.
\end{theorem}

\begin{proof}
	We consider two cases.
	
	Case 1. When $ k $ is even.
	
	Since $ \gcd(i,Q)=1 $, $ x^I $ is a permutation polynomial over $ U $. According to Lemma \ref{lemma 8}, the fractional polynomial (\ref{eq 27}) permutes $ U $. We denote the fractional polynomial (\ref{eq 27}) by $ \alpha_{3}(x) $, then $ \alpha_{3}(x^I) $ permutes $ U $, thus, let $ \beta \in U $, the equation
	\begin{equation*}
		\alpha_{3}(x^{I})=\beta
	\end{equation*}
	has only one solution in $ U $.
	
	For all $ x \in U $, $ x=\beta^{t}y $, where $ y \in U $ and $ t $ is an integer, thus the above equation can be written as
	\begin{equation*}
		\begin{split}
			&\sum_{w_{1}=1}^{K}\beta^{(K-w_{1})tI}y^{(K-w_{1})I}+\sum_{w_{2}=1}^{K}\beta^{(K-w_{2})tI}y^{(K-w_{2})I}\\+
			&\sum_{w_{3}=1}^{K}\left[\beta^{(K-w_{3})tI}+\beta^{(K-w_{3})tI+1}\right]y^{(K-w_{3})I}=0,
		\end{split}
	\end{equation*}
	where $ w_{1} \equiv 0 \pmod 3 $, $ w_{2} \equiv 2 \pmod 3 $ and $ w_{3} \equiv 1 \pmod 3 $.
	
	Since $ v_{2}(k) \leq v_{2}(m) $, $ \gcd(K-1,Q)=1 $. Since $ \gcd(i,Q)=1 $ and $ \gcd(K-1,Q)=1 $, there exists $ t \equiv \frac{1}{KI-I} \pmod Q $. Dividing both sides of the above equation by $ \beta^{(K-1)tI}y^{(K-1)i} $, then the equation
	\begin{equation}\label{eq 39}
		\begin{split}
			&g(y)+g(y)^{2^{m}}=0,\ where \\
			&g(y)=\sum_{j=1}^{K}\beta^{(1-j)tI}y^{(K-2j+1)i} \  with \ j \equiv 0,1 \pmod 3
		\end{split}
	\end{equation}
	has only one solution in $ U $.
	
	Case 2. When $ k $ is odd.
	
	Since $ \gcd(I,Q)=1 $, $ x^I $ is a permutation polynomial over $ U $. According to Lemma \ref{lemma 9}, the fractional polynomial (\ref{eq 31}) permutes $ U $. We denote the fractional polynomial (\ref{eq 31}) by $ \alpha_{4}(x) $, then $ \alpha_{4}(x^I) $ permutes $ U $, thus, let $ \beta \in U $, the equation
	\begin{equation*}
		\alpha_{4}(x^I)=\beta
	\end{equation*}
	has only one solution in $ U $.
	
	For all $ x \in U $, $ x=\beta^{t}y $, where $ y \in U $ and $ t $ is an integer, thus the above equation can be written as
	\begin{equation*}
		\begin{split}
			&\sum_{w_{1}=1}^{K}\beta^{(K-w_{1})tI}y^{(K-w_{1})I}+\sum_{w_{2}=1}^{K}\beta^{(K-w_{2})tI}y^{(K-w_{2})I}\\+
			&\sum_{w_{3}=1}^{K}\left[\beta^{(K-w_{3})tI}+\beta^{(K-w_{3})tI+1}\right]y^{(K-w_{3})I}=0,
		\end{split}
	\end{equation*}
	where $ w_{1} \equiv 1 \pmod 3 $, $ w_{2} \equiv 2 \pmod 3 $ and $ w_{3} \equiv 0 \pmod 3 $.
	
	Since $ v_{2}(k) \leq v_{2}(m) $, $ \gcd(K-1,Q)=1 $. Since $ \gcd(i,Q)=1 $ and $ \gcd(K-1,Q)=1 $, there exists $ t \equiv \frac{1}{KI-I} \pmod Q $. Dividing both sides of the above equation by $ \beta^{(K-1)tI}y^{(K-1)i} $, then the equation
	\begin{equation}\label{eq 42}
		\begin{split}
			&g(y)+g(y)^{2^{m}}=0,\ where \\ 
			&g(y)=\sum_{j=1}^{K}\beta^{(1-j)tI}y^{(K-2j+1)i} \ with \ j \equiv 0,1 \pmod 3
		\end{split}
	\end{equation}
	has only one solution in $ U $.
	
	We can find that the eqautions (\ref{eq 39}) and (\ref{eq 42}) are same. Let $ a $ be an integer, then $ \beta $ can be written as $ \delta^{(Q-2)a} $ where $ \delta \in \mathbb{F}_{2^{n}}^{*} $. According to the equations (\ref{eq 39}) and (\ref{eq 42}), the equation
	\begin{equation*}
		\begin{split}
			&g(\lambda)+g(\lambda)^{2^{m}}=0,\ where \\ &g(\lambda)=\sum_{j=1}^{K}\delta^{(1-j)(Q-2)atI}\lambda^{(K-2j+1)i} \ with \ j \equiv 0,1 \pmod 3	
		\end{split}
	\end{equation*}
	has only one solution in $ U $.
	
	Inspired by Lemma \ref{lemma 2} and Lemma \ref{lemma 3}, we hope that $ g(\lambda) $ has the form of $ \lambda^{d_{1}}+\sum_{j=2}^{K}\delta^{d_{1}-d_{j}}\lambda^{d_{j}} $. Similar to the proof of Theorem \ref{theorem 2}, we can conclude that the equation
	\begin{equation*}
		\lambda^{d_{1}}+\sum_{j=2}^{K}\delta^{d_{1}-d_{j}}\lambda^{d_{j}}+\left(\lambda^{d_{1}}+\sum_{j=2}^{K}\delta^{d_{1}-d_{j}}\lambda^{d_{j}}\right)^{2^m}=0
	\end{equation*}
	has only one solution in $ U $, where
	\begin{equation*}
		\begin{split}
			&d_{j} \equiv (K-1)i+uQ+(Q-2)(j-1)i \pmod {Q(Q-2)} \\
			&with \ j \equiv 0,1 \pmod 3.
		\end{split}
	\end{equation*}
	
	According to Lemma \ref{lemma 2} and Lemma \ref{lemma 3}, $ \sum_{j=1}^{K}x^{d_{j}} $ with $ j \equiv 0,1 \pmod 3 $ is a class of permutation polynomials over $ \mathbb{F}_{2^{2m}} $. For $ x \in \mathbb{F}_{2^{2m}} $, $ x^{d_{j}}=x^{d_{j} \pmod {Q(Q-2})} $ since $ Q(Q-2)=2^{2m}-1 $. Thus, the polynomial (\ref{eq 36}) is a class of permutation polynomials over $ \mathbb{F}_{2^{2m}} $. 
	
	The proof is complete.\qed
\end{proof}

The polynomial (\ref{eq 36}) is a large class of permutation polynomials. Some known classes of pemutation polynomials are the special cases of the class of permutation polynomials (\ref{eq 36}):

1. (\cite{Ref13} Theorem 3.6) $ k=3 $, $ u=-1 $, $ i=2^{m} $: 
\begin{equation*}
	x^{5}+x^{2^{m}+4}+x^{3\cdot2^{m}+2}+x^{4\cdot2^{m}+1}+x^{6\cdot2^{m}-1}.
\end{equation*}

2. (Theorem \ref{theorem 5}) $ k $ is odd, $ u=0 $, $ i=1 $, $ j=2^{k}+1-l $: 
\begin{equation*}
	\sum_{l=1}^{2^{k}}x^{2^{k+m}-l2^{m}+l-1}.
\end{equation*}

(i) (\cite{Ref13} Theorem 3.5) $ k=3 $: 
\begin{equation*}
	x^{7}+x^{2\cdot2^{m}+5}+x^{3\cdot2^{m}+4}+x^{5\cdot2^{m}+2}+x^{6\cdot2^{m}+1}.
\end{equation*}

\section{EA-inequivalence}\label{sec:4}	
In this section, we prove that, in all classes of permutation polynomials constructed in Section \ref{sec:3}, there exists a permutation polynomial which is EA-inequivalent to known permutation polynomials for all even positive integer $ m $.

\begin{theorem}\label{theorem 8}
	The known classes of permutation polynomials over $ \mathbb{F}_{2^{2m}} $ with a positive integer $ m $ are as follows:
	
	$ f_{1}(x)=x^{4}+x^{2^{m}+3}+x^{3\cdot2^{m}+1} \ (\gcd(m,3)=1) $ {\rm\cite{Ref3}}.
	
	$ f_{2}(x)=x^{2}+x^{2\cdot2^{m}}+x^{3\cdot2^{m}-1} \ (\gcd(m,3)=1) $ {\rm\cite{Ref3}}.
	
	$ f_{3}(x)=x^{5}+x^{2^{m}+4}+x^{4\cdot2^{m}+1} \ (\gcd(m,3)=1) $ {\rm\cite{Ref3}}.
	
	$ f_{4}(x)=x+x^{2^{m}}+x^{2^{2m-1}-2^{m-1}+1} \ (m \not\equiv 0 \pmod 3) $ {\rm\cite{Ref5}}.
	
	$ f_{5}(x)=x+ax^{2^{m+1}-1}+a^{2^{m-1}}x^{2^{2m}-2^{m}+1} $, where $ a \in \mathbb{F}_{2^{2m}} $ and the order of $ a $ is $ 2^{m}+1 $ {\rm\cite{Ref5}}.
	
	$ f_{6}(x)=x+x^{s(2^{m}-1)+1}+x^{t(2^{m}-1)+1} $, where $ k $ is a positive integer, $ k<m $, $ \gcd(2^{k}-1,2^{m}+1)=1 $ and $ (s,t)=(\frac{2^{k}}{2^{k}-1},\frac{-1}{2^{k}-1}) $ {\rm\cite{Ref6}}.
	
	$ f_{7}(x)=x+x^{s(2^{m}-1)+1}+x^{t(2^{m}-1)+1} $, where $ k $ is a positive integer, $ \gcd(2^{k}+1,2^{m}+1)=1 $ and $ (s,t)=(\frac{1}{2^{k}+1},\frac{2^{k}}{2^{k}+1}) $ {\rm\cite{Ref6}}.
	
	$ f_{8}(x)=x+x^{s(2^{m}-1)+1}+x^{t(2^{m}-1)+1} $, where $ m $ is even, $ (s,t)=(-\frac{1}{3},\frac{4}{3}) $ or $ (s,t)=(3,-1) $ or $ (s,t)=(-\frac{2}{3},\frac{5}{3}) $ {\rm\cite{Ref7}}.
	
	$ f_{9}(x)=x+x^{2^{m}(2^{m}-1)+1}+x^{2(2^{m}-1)+1} $ {\rm\cite{Ref11}}.
	
	$ f_{10}(x)=x+x^{2^{m-1}(2^{m}-1)+1}+x^{2^{2m-1}(2^{m}-1)+1} \ (m \ is \ even \ and \ m \geq 3) $ {\rm\cite{Ref12}}.
	
	$ f_{11}(x)=x^{5}+x^{2^{m}+4}+x^{3\cdot2^{m}+2}+x^{4\cdot2^{m}+1}+x^{5\cdot2^{m}} \ (m \not\equiv 0 \pmod 4) $ {\rm\cite{Ref13}}.
	
	$ f_{12}(x)=x^{5}+x^{2^{m}+4}+x^{2\cdot2^{m}+3}+x^{4\cdot2^{m}+1}+x^{5\cdot2^{m}} \ (m \not\equiv 0 \pmod 4) $ {\rm\cite{Ref13}}.
	
	$ f_{13}(x)=x^{7}+x^{2\cdot2^{m}+5}+x^{3\cdot2^{m}+4}+x^{5\cdot2^{m}+2}+x^{6\cdot2^{m}+1} \ (\gcd(m,3)=1) $ {\rm\cite{Ref13}}.
	
	$ f_{14}(x)=x^{5}+x^{2^{m}+4}+x^{3\cdot2^{m}+2}+x^{4\cdot2^{m}+1}+x^{6\cdot2^{m}-1} \ (m \equiv 2 \pmod 4) $ {\rm\cite{Ref13}}.
	
	$ f_{15}(x)=x^{5}+x^{3\cdot2^{m}+2}+x^{4\cdot2^{m}+1} \ (m \equiv 2 \pmod 4) $ {\rm\cite{Ref14}}.
	
	$ f_{16}(x)=x^{2^{m}+4}+x^{2^{m+1}+3}+x^{5\cdot2^{m}} \ (m \equiv 2 \pmod 4) $ {\rm\cite{Ref14}}.
	
	$ f_{17}(x)=x^{5}+x^{2^{m}+4}+x^{5\cdot2^{m}} \ (m \equiv 2 \pmod 4) $ {\rm\cite{Ref14}}.
	
	$ f_{18}(x)=x^{5}+x^{4\cdot2^{m}+1}+x^{5\cdot2^{m}} \ (m \equiv 2 \pmod 4) $ {\rm\cite{Ref14}}.
\end{theorem}

In the above known permutation polynomials, $ f_{6}(x) $ is constructed from the fractional permutation polynomial $ \frac{x^{2^{k}+1}+x^{2^{k}}+1}{x^{2^{k}+1}+x+1} $ over $ U $ by using Lemma \ref{lemma 1}, and $ f_{7}(x) $ is constructed from the fractional permutation polynomial $ \frac{x^{2^{k}+1}+x^{2^{k}}+x}{x^{2^{k}}+x+1} $ over $ U $ by using Lemma \ref{lemma 1}. Obviously, the fractional permutation polynomials (\ref{eq 1}) and (\ref{eq 27}) are EA-equivalent to $ \frac{x^{2^{k}+1}+x^{2^{k}}+1}{x^{2^{k}+1}+x+1} $, and the fractional permutation polynomials (\ref{eq 14}) and (\ref{eq 31}) are EA-equivalent to $ \frac{x^{2^{k}+1}+x^{2^{k}}+x}{x^{2^{k}}+x+1} $. However, the next proof will show that, in new classes of permutation polynomials, there exists a permutation polynomial which is EA-inequivalent to $ f_{6}(x) $ and $ f_{7}(x) $ for all even positive integer $ m $.

\begin{theorem}\label{theorem 7}
	In the class of permutation polynomials {\rm(\ref{eq 4})}, there exists a permutation polynomial which is EA-inequivalent to all the known permutation polynomials in Theorem {\rm\ref{theorem 8}} for all even positive integer $ m $. 
\end{theorem}

\begin{proof}
	According to the property that EA-equivalent nonconstant functions have same algebraic degree, two polynomials are EA-inequivalent if they have different algebraic degree.	
	
	For the class of permutation polynomials (\ref{eq 4}), let $ k=1,s=2^{2m-1}-3 $, we can get the polynomial $ g_{1}(x)=x^{2^{2m}-2^{m+1}-1}+x^{2^{2m-1}+2^{m-1}-3}+x^{2^{m+1}+2^{m}-5}+x^{2^{2m}-2^{m+2}+1}+x^{2^{2m}-2^{m+2}-2^{m}+2}+x^{2^{2m-1}-2^{m+1}+2^{m-1}-1}+x^{2^{2m-1}-2^{m+1}-2^{m-1}}+x^{2^{m}-3}+x^{2^{2m}-3} $, and the algebraic degree of $ g_{1}(x) $ is $ 2m-1 $. When $ m=2 $, let $ k=1,s=1 $, we can obtain the polynomial $ g_{2}(x)=x^{2} $, and the algebraic degree of $ g_{2}(x) $ is 1.
	
	The algebraic degree of $ f_{3}(x) $, $ f_{17}(x) $ and $ f_{18}(x) $ is less than or equal to 2. Since $ m \geq 2 $, $ 2m-1 \geq 3 $, therefore, $ g_{1}(x) $ has different algebraic degree from $ f_{3}(x) $, $ f_{17}(x) $ and $ f_{18}(x) $.
	
	The algebraic degree of $ f_{1}(x) $, $ f_{11}(x) $, $ f_{12}(x) $, $ f_{13}(x) $, $ f_{15}(x) $ and $ f_{16}(x) $ is equal to 3. When $ m > 2 $, $ 2m-1 > 3 $, therefore, $ g_{1}(x) $ has different algebraic degree from $ f_{1}(x) $, $ f_{11}(x) $, $ f_{12}(x) $, $ f_{13}(x) $, $ f_{15}(x) $ and $ f_{16}(x) $. When $ m = 2 $, $ g_{2}(x) $ has different algebraic degree from $ f_{1}(x) $, $ f_{11}(x) $, $ f_{12}(x) $, $ f_{13}(x) $, $ f_{15}(x) $ and $ f_{16}(x) $.
	
	The algebraic degree of $ f_{2}(x) $, $ f_{4}(x) $, $ f_{5}(x) $, $ f_{9}(x) $ and $ f_{10}(x) $ is equal to $ m+1 $. When $ m > 2 $, $ 2m-1 > m+1 $, therefore, $ g_{1}(x) $ has different algebraic degree from $ f_{2}(x) $, $ f_{4}(x) $, $ f_{5}(x) $, $ f_{9}(x) $ and $ f_{10}(x) $. When $ m = 2 $, $ m+1=3 $, therefore, $ g_{2}(x) $ has different algebraic degree from $ f_{2}(x) $, $ f_{4}(x) $, $ f_{5}(x) $, $ f_{9}(x) $ and $ f_{10}(x) $.
	
	When $ m \ge 2 $, the algebraic degree of $ f_{14}(x) $ is $ m+2 $, and $ 2m-1 \neq m+2 $ since $ m $ is even, therefore, $ g_{1}(x) $ has different algebraic degree from $ f_{14}(x) $. When $ m = 2 $, the algebraic degree of $ f_{14}(x) $ is 3, therefore, $ g_{2}(x) $ has different algebraic degree from $ f_{14}(x) $.
	
	$ f_{6}(x) $, $ f_{7}(x) $ and $ f_{8}(x) $ have the same form. For all $ 1 \leq s \leq 2^{m} $, $ s $ can be written as $ \sum_{i=j+1}^{m-1}s_{i}2^{i}+2^{j} $, where $ i,j \in \mathbb{N} $, $ s_{i} \in \{0,1\} $ and $ j $ is the minimum $ i $ such that $ s_{i} \neq 0 $. Thus, $ s(2^{m}-1)=\sum_{i=j+1}^{m-1}s_{i}2^{i+m}+2^{j}(2^{m}-1-\sum_{i=j+1}^{m-1}s_{i}2^{i-j}) $. Denoting the Hamming weight by $ wt_{2}(\cdot) $, it is obvious that $ wt_{2}(\sum_{i=j+1}^{m-1}s_{i}2^{i+m})=wt_{2}(s)-1 $ and $ wt_{2}(2^{m}-1-\sum_{i=j+1}^{m-1}s_{i}2^{i-j})=m-(wt_{2}(s)-1) $. Since $ 2^{j+m+1} > 2^{j+m} $, $ wt_{2}(s(2^{m}-1))=wt_{2}(s)-1+m-(wt_{2}(s)-1)=m $. Thus, $ wt_{2}(s(2^{m}-1)+1) \leq m+1 $. That is, the algebraic degree of $ f_{6}(x) $, $ f_{7}(x) $ and $ f_{8}(x) $ is less than or equal to $ m+1 $. When $ m > 2 $, $ 2m-1 > m+1 $, therefore, $ g_{1}(x) $ has different algebraic degree from $ f_{6}(x) $, $ f_{7}(x) $ and $ f_{8}(x) $. When $ m = 2 $, the Hamming weight of $ s(2^{m}-1)+1 $ is equal to 1 only for $ s=1 $, thus, the algebraic degree of $ f_{6}(x) $, $ f_{7}(x) $ and $ f_{8}(x) $ is not equal to 1 since $ s \neq t $, therefore, $ g_{2}(x) $ has different algebraic degree from $ f_{6}(x) $, $ f_{7}(x) $ and $ f_{8}(x) $.
	
	In conclusion, in the class of permutation polynomials (\ref{eq 4}), there exists a  permutation polynomial which is EA-inequivalent to all the known permutation polynomials in Theorem \ref{theorem 8} for all even positive integer $ m $.
	
	The proof is complete.\qed
\end{proof}

Similarly, for the class of permutation polynomials (\ref{eq 5}), let $ k=1,s=-2,u=0,i=1 $, we can get the polynomial $ g_{3}(x)=x^{2^{2m}-2}+x^{2^{2m}-2^{m+1}}+x^{2^{2m}-2^{m+2}+2}+x^{2^{2m}-3\cdot2^{m}+1}+x^{3\cdot2^{m}-4} $, and the algebraic degree of $ g_{3}(x) $ is $ 2m-1 $. When $ m=2 $, let $ k=1,s=-1,u=4,i=1 $, we can obtain the polynomial $ g_{4}(x)=x^{4} $, and the algebraic degree of $ g_{4}(x) $ is 1. Thus, we can prove that, in the class of permutation polynomials (\ref{eq 5}), there exists a permutation polynomial which is  EA-inequivalent to all the known permutation polynomials in Theorem \ref{theorem 8} for all even positive integer $ m $ by using the same method in the proof of Theorem \ref{theorem 7}.

For the class of permutation polynomials (\ref{eq 17}), let $ k=1,s=2^{2m-1}-3 $, we can get the polynomial $ g_{5}(x)=x^{2^{2m}-2^{m+2}+2^{m}}+x^{2^{2m-1}-2^{m-1}-2}+x^{2^{m+1}-4}+x^{2^{2m}-2^{m+2}+1}+x^{2^{2m}-2^{m+2}-2^{m}+2}+x^{2^{2m-1}-2^{m+1}-2^{m-1}}+x^{2^{2m-1}-2^{m+1}+2^{m-1}-1}+x^{2^{m}-3}+x^{2^{2m}-3} $, and the algebraic degree of $ g_{5}(x) $ is $ 2m-1 $. When $ m=2 $, let $ k=3,s=1 $, we can obtain the polynomial $ g_{6}(x)=x^{8} $, and the algebraic degree of $ g_{6}(x) $ is 1. Thus, we can prove that, in the class of permutation polynomials (\ref{eq 17}), there exists a permutation polynomial which is EA-inequivalent to all the known permutation polynomials in Theorem \ref{theorem 8} for all even positive integer $ m $ by using the same method in the proof of Theorem \ref{theorem 7}.

For the class of permutation polynomials (\ref{eq 18}), let $ k=1,s=-2,u=0,i=1 $, we can get the polynomial $ g_{7}(x)=x^{2^{m}-2}+x^{2^{2m}-2^{m}-1}+x^{2^{2m}-2^{m+2}+2}+x^{2^{2m}-3\cdot2^{m}+1}+x^{2^{m+1}-3} $, and the algebraic degree of $ g_{7}(x) $ is $ 2m-1 $. When $ m=2 $, let $ k=3,s=-1,u=2,i=1 $, we can obtain the polynomial $ g_{8}(x)=x^{4} $, and the algebraic degree of $ g_{8}(x) $ is 1. Thus, we can prove that, in the class of permutation polynomials (\ref{eq 18}), there exists a permutation polynomial which is EA-inequivalent to all the known permutation polynomials in Theorem \ref{theorem 8} for all even positive integer $ m $ by using the same method in the proof of Theorem \ref{theorem 7}.

For the class of permutation polynomials (\ref{eq 35}), let $ k=2m-1 $, we can get the polynomial $ g_{9}(x)=\sum_{j=1}^{2^{2m-1}}x^{2^{m-1}-j\cdot2^{m}+j-1} $ with $ j \equiv 0,2 \pmod 3 $, and the algebraic degree of $ g_{9}(x) $ is $ 2m-1 $. When $ m=2 $, let $ k=1 $, we can obtain the polynomial $ g_{10}(x)=x $, and the algebraic degree of $ g_{10}(x) $ is 1. Thus, we can prove that, in the class of permutation polynomials (\ref{eq 35}), there exists a permutation polynomial which is EA-inequivalent to all the known permutation polynomials in Theorem \ref{theorem 8} for all even positive integer $ m $ by using the same method in the proof of Theorem \ref{theorem 7}.

For the class of permutation polynomials (\ref{eq 36}), let $ k=2,u=-2,i=1 $, we can get the polynomial $ g_{11}(x)=x^{2^{2m}-2^{m+1}}+x^{2^{2m}-2}+x^{2^{m}-2} $, and the algebraic degree of $ g_{11}(x) $ is $ 2m-1 $. When $ m=2 $, let $ k=1,u=0,i=1 $, we can obtain the polynomial $ g_{12}(x)=x $, and the algebraic degree of $ g_{12}(x) $ is 1. Thus, we can prove that, in the class of permutation polynomials (\ref{eq 36}), there exists a permutation polynomial which is EA-inequivalent to all the known permutation polynomials in Theorem \ref{theorem 8} for all even positive integer $ m $ by using the same method in the proof of Theorem \ref{theorem 7}.

So far, for all even positive integer $ m $, we have proved that, in all classes of permutation polynomials constructed in Section \ref{sec:3}, there exists a permutation polynomial which is EA-inequivalent to the known classes of permutation polynomials in Theorem \ref{theorem 8}.

Furthermore, from the above proof, it is obvious that, in all new permutation polynomials, there exists a permutation polynomial with algebraic degree $ 2m-1 $, which is the maximum algebraic degree of permutation polynomials over $ \mathbb{F}_{2^{2m}} $.

\section{Conclusion}\label{sec:5}	
In this paper, from new fractional permutation polynomials over $ U $, six new classes of permutation polynomials over $ \mathbb{F}_{2^{2m}} $ with coefficients 1 have been constructed. Three new classes of permutation polynomials are the special cases of other three new classes of permutation polynomials. Furthermore, we have proved that, in all new permutation polynomials, there exists a permutation polynomial which is EA-inequivalent to known permutation polynomials for all even positive integer $ m $.

\subsubsection{Acknowledgements} This work was supported by Beijing Natural Science Foundation (No. 4202037), NSF of China with contract (No.61972018).

\end{document}